\documentclass[12pt]{amsart}

\usepackage[english]{babel}
\usepackage{microtype}

\usepackage{amsmath,amssymb,amsfonts,amsthm,mathtools}

\usepackage{xcolor}

\usepackage{todonotes}
\usepackage[colorlinks,linkcolor=blue,anchorcolor=blue,citecolor=blue,backref=page]{hyperref}
\hypersetup{breaklinks=true}

\usepackage[norefs,nocites]{refcheck}  

\allowdisplaybreaks

\theoremstyle{plain}
\newtheorem{theorem}{Theorem}[section]
\newtheorem{proposition}[theorem]{Proposition}
\newtheorem{lemma}[theorem]{Lemma}
\newtheorem{corollary}[theorem]{Corollary}

\theoremstyle{remark}

\numberwithin{equation}{section}

\newcommand{\om}{\omega}
\newcommand{\one}{\mathbf 1}

\title[Selberg sieve weights and Kloosterman sign changes, I]
{Selberg sieve weights and sign changes of Kloosterman sums. I.
Uniform asymptotics for shifted weights}

\author{Yixiu Xiao}
\address{School of Mathematical Sciences, Shanghai Jiao Tong University,
800 Dongchuan Road, Shanghai 200240, China}
\email{yixiuxiao98@gmail.com}

\author{Hongze Li}
\address{School of Mathematical Sciences, Shanghai Jiao Tong University,
800 Dongchuan Road, Shanghai 200240, China}
\email{lihz@sjtu.edu.cn}
\thanks{Hongze Li is the corresponding author.}


\begin{document}

\begin{abstract}
We prove an asymptotic formula for smoothed correlations of divisor sums with independently shifted quadratic cutoff functions. With \(Y\) denoting the averaging scale, \(R\) the divisor cutoff, and \(g\) a fixed smooth averaging weight, the error is \(O_g(Y/(\log Y)^2)\), uniformly for \(Y^{1/5}\leq R\leq Y^{1/3}\) and shifts \(s,u\in[0,1/2]\), including coincident shifts. The main term is an explicit bilinear form in the first and second derivatives of the cutoff functions. The proof uses a fixed zero-free half-plane for the Riemann zeta-function to evaluate the relevant Mellin integrals by successive contour shifts and residue calculations. We also obtain a mean-square estimate for differences of shifted divisor sums, which is used in Part II to control small prime factors in the study of sign changes of Kloosterman sums.
\end{abstract}

\maketitle

\section{Introduction}
A Selberg sieve weight is the square of a truncated divisor sum.
In applications, separating a prime factor can change its cutoff
function. For the quadratic cutoff functions considered here, writing a
square-free modulus as \(pn\) replaces the cutoff function in the divisor sum
by the difference between it and a translate through
\(\log p/\log R\). This shift may tend to zero with the main
parameter. To estimate the resulting divisor sums, we need a
correlation formula whose error is uniform in the shifts.

We prove such a formula for shifted quadratic cutoff functions,
with an explicit main term and error \(O_g(Y/(\log Y)^2)\).
It remains valid when the shifts approach one another or coincide.
As a consequence, we obtain a mean-square bound for the difference
between the unshifted divisor sum and a shifted divisor sum. Part~II uses
this bound to control small prime factors in the study of
Kloosterman sign changes for square-free moduli with at most four
prime factors. The present paper establishes the divisor-sum
estimate; the choice of cutoff and its use in the sign-change
argument belong to Part~II.

\subsection{Main results}

Fix a nonnegative function \(g\in C_c^\infty((0,\infty))\), supported
in \([1,2]\), and normalized by
\[
  \int_0^\infty g(x)\,dx=1.
\]
Let \(Y>e\) and \(R>1\), and put
\[
  L=\log R,
  \qquad
  T=\log Y,
  \qquad
  \vartheta=\frac{T}{L}.
\]
Thus \(Y^{1/5}\leq R\leq Y^{1/3}\) is equivalent to
\(3\leq\vartheta\leq5\).  Let \(\mu\) be the Möbius function,
let \(\om(n)\) be the number of distinct prime divisors of \(n\), and
write \(\one_E\) for the indicator of a condition \(E\).  All sums over
\(m\) below are over the positive integers.
The notation \(O_g(\cdot)\) and \(\ll_g\) allows the implied constant
to depend on \(g\), but not on \(Y,R,s\), or \(u\).

For \(v\in\mathbb R\), write \(v_+=\max\{v,0\}\).  For
\(0\leq s<1\), define the shifted quadratic profile
\[
  G_s(x)=(x-s)_+^2
\]
and the associated divisor sum
\[
  \Lambda_{s,R}(m)
  =
  \sum_{d\mid m}
    \mu(d)G_s\!\left(\frac{\log(R/d)}{\log R}\right)
    \one_{d\leq R}.
\]
The effective divisor cutoff is \(R^{1-s}\); equivalently,
\[
  \Lambda_{s,R}(m)
  =
  \frac1{L^2}
  \sum_{\substack{d\mid m\\d\leq R^{1-s}}}
    \mu(d)\bigl(\log(R^{1-s}/d)\bigr)^2.
\]

Let \(W^{2,2}(0,1)\) be the Sobolev space of functions whose weak
derivatives through order two lie in \(L^2(0,1)\).  For
\(U,V\in W^{2,2}(0,1)\), set
\begin{equation}\label{eq:bilinear-form}
\begin{aligned}
  \mathcal B_\vartheta(U,V)
  &:={}
  2\vartheta\int_0^1U'(x)V'(x)\,dx\\
  &\quad+
  \vartheta\int_0^1U''(x)V''(x)
    (1-x)(\vartheta-1+x)\,dx.
\end{aligned}
\end{equation}
Here the derivatives are understood in the weak sense.  In particular,
\(G_s\in W^{2,2}(0,1)\), with
\[
\begin{aligned}
  G_s'(x)&=2(x-s)_+,\\
  G_s''(x)&=2\one_{(s,1)}(x)
  \quad\text{almost everywhere on }(0,1).
\end{aligned}
\]

\begin{theorem}
\label{thm:kernel}
As \(Y\to\infty\), uniformly for
\[
  Y^{1/5}\leq R\leq Y^{1/3},
  \qquad
  s,u\in[0,1/2],
\]
one has
\begin{equation*}
\begin{aligned}
  &\sum_m
    g\!\left(\frac mY\right)
    \mu^2(m)2^{\om(m)}
    \Lambda_{s,R}(m)\Lambda_{u,R}(m)\\
  &\qquad=
  \frac{Y}{\log Y}\mathcal B_\vartheta(G_s,G_u)
  +O_g\!\left(\frac{Y}{(\log Y)^2}\right).
\end{aligned}
\end{equation*}
\end{theorem}

The main term can be written without integrals.  If
\[
  \alpha=1-s,\qquad \beta=1-u,\qquad
  \gamma=\min\{\alpha,\beta\}=1-\max\{s,u\},
\]
then
\begin{equation}\label{eq:explicit-kernel}
  \mathcal B_\vartheta(G_s,G_u)
  =
  4\vartheta\left\{
    2\alpha\beta\gamma+
    \left(\frac{\vartheta}{2}-\alpha-\beta\right)\gamma^2+
    \frac{\gamma^3}{3}
  \right\}.
\end{equation}
In the unshifted case,
\(\mathcal B_\vartheta(G_0,G_0)=2\vartheta^2+
\frac43\vartheta\).

Xi's sieve asymptotic \cite[Lemma~2.1]{XiCorr}, in the setting of
\cite{Xi2}, gives, for fixed smooth cutoff functions satisfying its
hypotheses, a
correlation formula by polarization, with an \(o(Y/\log Y)\) error.
That statement alone does not give an error uniform in cutoff
functions that vary with \(Y\). Theorem~\ref{thm:kernel} supplies
such an error for the present family, whose second derivatives have
jumps at the shift parameters.

A useful consequence is obtained by subtracting the shifted profile
from the unshifted one.  For \(0\leq t\leq1/2\), define
\(H_t=G_0-G_t\) and
\[
  W_{H_t,R}(m)
  =
  \left(
    \sum_{d\mid m}
      \mu(d)H_t\!\left(\frac{\log(R/d)}{\log R}\right)
      \one_{d\leq R}
  \right)^2.
\]
By linearity in the profile,
\[
  W_{H_t,R}(m)
  =
  \bigl(\Lambda_{0,R}(m)-\Lambda_{t,R}(m)\bigr)^2.
\]
The following result therefore follows by applying
Theorem~\ref{thm:kernel} to the shift pairs \((0,0)\), \((0,t)\), and
\((t,t)\).

\begin{corollary}
\label{cor:shifted-profile}
Uniformly for
\[
  Y^{1/5}\leq R\leq Y^{1/3},
  \qquad
  0\leq t\leq1/2,
\]
one has
\begin{equation}\label{eq:shifted-profile-asymptotic}
  \sum_m
    g\!\left(\frac mY\right)\mu^2(m)2^{\om(m)}W_{H_t,R}(m)
  =
  \frac{Y}{\log Y}\mathcal Q_\vartheta(H_t)
  +O_g\!\left(\frac{Y}{(\log Y)^2}\right),
\end{equation}
where
\begin{equation}\label{eq:shifted-profile-main-term}
  \mathcal Q_\vartheta(H_t)
  :=
  \mathcal B_\vartheta(H_t,H_t)
  =
  \vartheta\left\{
    4(\vartheta-1)t+(12-2\vartheta)t^2
    -\frac{20}{3}t^3
  \right\}.
\end{equation}
In particular,
\(\mathcal Q_\vartheta(H_t)=4\vartheta(\vartheta-1)t+O(t^2)\)
as \(t\to0\), uniformly for \(3\leq\vartheta\leq5\).
Consequently,
\begin{equation}\label{eq:shifted-profile-bound}
  \sum_m
    g\!\left(\frac mY\right)\mu^2(m)2^{\om(m)}W_{H_t,R}(m)
  \ll_g
  \left(t+\frac1{\log Y}\right)\frac{Y}{\log Y}.
\end{equation}
\end{corollary}

\subsection{Small prime factors and the proof}

Let \(p\leq R^{1/2}\) be prime. When a square-free modulus is
written as \(pn\), with \(p\nmid n\), put
\[
  t=\frac{\log p}{\log R}.
\]
Every divisor of \(pn\) is uniquely of the form \(e\) or \(pe\), with
\(e\mid n\).  Pairing the corresponding terms in the definition of
\(\Lambda_{0,R}(pn)\), and using
\[
  \mu(pe)=-\mu(e)
  \qquad\text{and}\qquad
  G_0(x-t)=G_t(x),
\]
gives the exact identity
\[
  \Lambda_{0,R}(pn)
  =
  \Lambda_{0,R}(n)-\Lambda_{t,R}(n).
\]
Corollary~\ref{cor:shifted-profile} therefore yields an estimate
for the mean square of \(\Lambda_{0,R}(pn)\), recorded as
Corollary~\ref{cor:small-prime} below. When this is applied to
moduli of size \(X\), the variable \(n\) has scale \(Y=X/p\);
the range \(Y^{1/5}\leq R\leq Y^{1/3}\) must be checked at that
scale. Part~II makes these parameter choices when summing over
small primes.

The proof uses the fixed zero-free half-plane
\(\Re z>7/8\) for the Riemann zeta-function established in
\cite[Theorem~1.1]{OpenAI2026}. Only a fixed gap from the line
\(\Re z=1\) is needed; the numerical value of the boundary is not
essential to the argument.
In
Section~\ref{sec:arithmetic} we obtain a Mellin representation and
move the contour in the outer variable, keeping the two other
contours in the right half-plane. The inequality
\(\vartheta>\alpha+\beta\), where \(\alpha=1-s\) and
\(\beta=1-u\), permits a shift that crosses only the pole at
\(z=1\). In Section~\ref{sec:residues} we move the remaining
contours within the fixed zero-free half-plane. The residue on
\(w_1+w_2=0\) contributes only to the error, uniformly even when
the shifts coincide, and the iterated residue at the origin gives
the main term. The corollaries are proved in
Section~\ref{sec:consequences}.

\section{Mellin inversion and contour integration}
\label{sec:arithmetic}

Write
\begin{align*}
  \widetilde g(z)&=\int_0^\infty g(x)x^{z-1}\,dx,\\
  \mathfrak C_{s,u}(Y,R)
  &=\sum_m g\!\left(\frac mY\right)\mu^2(m)2^{\om(m)}
    \Lambda_{s,R}(m)\Lambda_{u,R}(m).
\end{align*}
The function \(\widetilde g\) is entire, \(\widetilde g(1)=1\), and,
on every fixed vertical strip, \(\widetilde g\) and each of its
derivatives decay faster than any power of the imaginary part.
Throughout the proof we put
\[
  \alpha=1-s,\qquad \beta=1-u,
  \qquad \frac12\leq\alpha,\beta\leq1,
  \qquad 3\leq\vartheta\leq5.
\]
In particular, \(\vartheta-\alpha-\beta\geq1\).
All vertical contours are oriented upwards; \(\int_{(c)}\) denotes
integration on the line with real part \(c\).

\subsection{The Euler product}

\begin{lemma}\label{lem:arithmetic}
For \(c_z>1\) and \(c_1,c_2>0\), one has the absolutely convergent
representation
\begin{equation}\label{eq:triple-mellin}
\begin{aligned}
\mathfrak C_{s,u}(Y,R)
={}&\frac{4}{L^4}\frac1{(2\pi i)^3}
\int_{(c_1)}\int_{(c_2)}\int_{(c_z)}
\widetilde g(z)Y^z
\frac{\zeta(z)^2\zeta(z+w_1+w_2)^2}
     {\zeta(z+w_1)^2\zeta(z+w_2)^2}\\
&\hspace{25mm}\times
\mathfrak A(z,w_1,w_2)
\frac{e^{L(\alpha w_1+\beta w_2)}}{w_1^3w_2^3}
\,dz\,dw_2\,dw_1.
\end{aligned}
\end{equation}
The factor \(\mathfrak A\) is holomorphic in the region
\begin{equation}\label{eq:arithmetic-tube}
\begin{split}
\Omega=\bigl\{(z,w_1,w_2):{}&
\min\{\Re z,\Re(z+w_1),\Re(z+w_2),\\
&\hspace{12mm}\Re(z+w_1+w_2)\}>\tfrac12\bigr\}.
\end{split}
\end{equation}
It and all its fixed-order derivatives are bounded on each closed
region obtained by replacing \(>1/2\) with \(\geq1/2+\delta\),
where \(\delta>0\). Moreover,
\begin{equation}\label{eq:arithmetic-axes}
  \mathfrak A(z,w_1,0)=\mathfrak A(z,0,w_2)=1
\end{equation}
whenever the arguments belong to \(\Omega\).
\end{lemma}

\begin{proof}
The identity
\[
  \frac{(\log X)_+^2}{2}
  =\frac1{2\pi i}\int_{(c)}\frac{X^w}{w^3}\,dw
  \qquad(X>0,\ c>0)
\]
gives
\[
\Lambda_{s,R}(m)
=\frac{2}{L^2}\frac1{2\pi i}\int_{(c)}
  \frac{e^{\alpha Lw}}{w^3}
  \prod_{p\mid m}(1-p^{-w})\,dw.
\]
After Mellin inversion of \(g\), the resulting Dirichlet series is
\begin{equation}\label{eq:direct-euler}
\begin{aligned}
\mathcal F(z,w_1,w_2)
&=\sum_{m\geq1}\frac{\mu^2(m)2^{\om(m)}}{m^z}
  \prod_{p\mid m}(1-p^{-w_1})(1-p^{-w_2})\\
&=\prod_p\{1+2p^{-z}(1-p^{-w_1})(1-p^{-w_2})\}.
\end{aligned}
\end{equation}
For \(\Re z=c_z>1\) and \(\Re w_j=c_j>0\), the absolute
Dirichlet series is bounded, uniformly in all imaginary parts, by
\[
\prod_p\{1+2p^{-c_z}(1+p^{-c_1})(1+p^{-c_2})\}<\infty.
\]
Together with the factors \(w_1^{-3}w_2^{-3}\) and the rapid decay of
\(\widetilde g\), this justifies the interchanges of sums and
integrals.

Set
\[
q_0=p^{-z},\qquad q_1=p^{-z-w_1},\qquad
q_2=p^{-z-w_2},\qquad q_{12}=p^{-z-w_1-w_2}.
\]
Extracting the zeta-factors in~\eqref{eq:triple-mellin} from
\eqref{eq:direct-euler} gives \(\mathfrak A=\prod_p\mathfrak A_p\),
where
\begin{equation}\label{eq:arithmetic-local}
\mathfrak A_p
=\frac{(1-q_0)^2(1-q_{12})^2}{(1-q_1)^2(1-q_2)^2}
  \{1+2(q_0-q_1-q_2+q_{12})\}.
\end{equation}
Put \(Q=\max\{|q_0|,|q_1|,|q_2|,|q_{12}|\}\). Near the origin,
\[
\frac{(1-q_0)^2(1-q_{12})^2}{(1-q_1)^2(1-q_2)^2}
=1-2(q_0-q_1-q_2+q_{12})+O(Q^2).
\]
Multiplication by the last factor in~\eqref{eq:arithmetic-local}
therefore gives \(\mathfrak A_p=1+O(Q^2)\).
On the closed region with margin \(\delta\), all four
\(q\)-variables have modulus at most \(p^{-1/2-\delta}\), and
the denominators are bounded away from zero. Consequently,
\[
\mathfrak A_p-1\ll_\delta p^{-1-2\delta}.
\]
The Euler product converges normally in the region~\eqref{eq:arithmetic-tube}
and is bounded
on every such closed region. Cauchy's estimates on polydiscs of radius
\(\delta/12\) give the asserted derivative bounds. The identity
initially obtained in the region of absolute convergence therefore
continues meromorphically to \(\Omega\).

In~\eqref{eq:arithmetic-local}, if \(w_2=0\), then \(q_2=q_0\)
and \(q_{12}=q_1\), so that
\(\mathfrak A_p=1\). The same argument applies when \(w_1=0\).
This proves~\eqref{eq:arithmetic-axes}.
\end{proof}

\subsection{The contour shift}

It suffices to consider sufficiently large \(L\). Fix
\[
  a=\frac1{100}.
\]
Put \(v=z-1\) and \(W=w_1+w_2\), and define
\begin{equation}\label{eq:residue-B}
\begin{aligned}
  \mathcal H(v;w_1,w_2)
  :={}&\widetilde g(1+v)\mathfrak A(1+v,w_1,w_2)\\
  &\times
  \frac{\zeta(1+v+W)^2}
       {\zeta(1+v+w_1)^2\zeta(1+v+w_2)^2}.
\end{aligned}
\end{equation}
Let \(\gamma_0\) denote Euler's constant, and set
\begin{equation}\label{eq:residue-density}
  \mathcal R_L(w_1,w_2)
  :=\frac{(\vartheta L+2\gamma_0)\mathcal H(0;w_1,w_2)
             +\partial_v\mathcal H(0;w_1,w_2)}
            {w_1^3w_2^3}.
\end{equation}

\begin{lemma}\label{lem:single-residue}
Uniformly in the stated parameter ranges,
\begin{equation}\label{eq:single-residue}
\begin{aligned}
  \mathfrak C_{s,u}(Y,R)
  ={}&\frac{4Y}{L^4}\frac1{(2\pi i)^2}
  \int_{(a)}\int_{(a)}
  e^{L(\alpha w_1+\beta w_2)}
  \mathcal R_L(w_1,w_2)\,dw_1\,dw_2\\
  &+O_g\!\left(YL^{-4}e^{-L/40}\right).
\end{aligned}
\end{equation}
The double integral is absolutely convergent.
\end{lemma}

\begin{proof}
In~\eqref{eq:triple-mellin}, take \(c_z=2\) and
\(c_1=c_2=a_0=1/16\). Move the contour of \(v=z-1\) to
\(\Re v=-b\), where \(b=1/20\), keeping the \(w_1\)- and \(w_2\)-contours
fixed. Throughout this deformation,
\[
\begin{aligned}
  \Re(1+v+w_j)&\geq1+a_0-b=1+\frac1{80},\\
  \Re(1+v+W)&\geq1+2a_0-b=1+\frac3{40}.
\end{aligned}
\]
The zeta quotient in~\eqref{eq:residue-B} is therefore uniformly
bounded.  The arithmetic factor is bounded by
Lemma~\ref{lem:arithmetic}, since all four real parts defining its
region are at least \(1-b>1/2\).  The pole \(v=-W\) has real part
\(-2a_0<-b\), and hence the only pole crossed is the double pole at \(v=0\).

The Mellin transform \(\widetilde g(1+v)\) decays faster than any
power of \(\lvert\Im v\rvert\), uniformly on the fixed strip
\(-b\leq\Re v\leq1\).  Together with polynomial growth of zeta
on a fixed vertical strip, this bounds the integrals over the
horizontal \(v\)-segments by a quantity tending to zero, times
\(\lvert w_1w_2\rvert^{-3}\).  Since
\(\int_{\mathbb R}\lvert a_0+it\rvert^{-3}\,dt<\infty\), the
residue theorem may be integrated over both \(w_1\)- and \(w_2\)-contours.
The integral on the new \(v\)-line is
\[
  \ll_g\frac{Y}{L^4}
  e^{L\{-b\vartheta+a_0(\alpha+\beta)\}}
  \ll_g\frac{Y}{L^4}e^{-L/40},
\]
because \(b\vartheta-a_0(\alpha+\beta)\geq3/20-1/8=1/40\).
The expansion
\(\zeta(1+v)^2=v^{-2}+2\gamma_0v^{-1}+O(1)\)
shows that
\[
\mathop{\rm Res}_{v=0}
 e^{\vartheta Lv}\zeta(1+v)^2\mathcal H(v;w_1,w_2)
=(\vartheta L+2\gamma_0)\mathcal H(0;w_1,w_2)
 +\partial_v\mathcal H(0;w_1,w_2).
\]
After division by \(w_1^3w_2^3\), this is
\eqref{eq:residue-density}.

Finally, move the contours in \(w_1\) and \(w_2\) from \(a_0\)
to \(a\), one at a time. The functions \(\mathcal H(0;w_1,w_2)\)
and \(\partial_v\mathcal H(0;w_1,w_2)\) are holomorphic when
\(\Re w_1,\Re w_2>0\). On the fixed strips involved, absolute
convergence of the zeta and reciprocal-zeta Dirichlet series and
of their first derivatives gives
\[
  \mathcal H(0;w_1,w_2),\quad
  \partial_v\mathcal H(0;w_1,w_2)\ll_g1.
\]
Thus \(\mathcal R_L\ll_g(L+1)|w_1w_2|^{-3}\). This proves
absolute convergence and makes the horizontal edges tend to zero
for each fixed \(L\). No pole is crossed, proving
\eqref{eq:single-residue}.
\end{proof}

\section{Evaluation by successive residues}\label{sec:residues}

\subsection{Bounds in a fixed zero-free half-plane}

By \cite[Theorem~1.1]{OpenAI2026}, the Riemann zeta-function has
no zeros in \(\Re z>7/8\). We record the standard growth consequence
needed for the contour shifts.

\begin{lemma}\label{lem:fixed-strip-growth}
Fix \(7/8<\sigma_0<1\), an integer \(j\geq0\), and
\(\epsilon>0\). Uniformly for
\(\sigma_0\leq\sigma\leq2\) and \(|t|\geq3\),
\[
  |\zeta^{(j)}(\sigma+it)|
  +\left|\left(\frac1\zeta\right)^{(j)}(\sigma+it)\right|
  \ll_{\sigma_0,j,\epsilon}(1+|t|)^\epsilon.
\]
\end{lemma}

\begin{proof}
First let \(j=0\). Put \(r=2-\sigma_0\), and choose fixed radii
\[
  \frac12<r<R_1<R_2<\frac98.
\]
The disc of radius \(R_2\) about \(2+it\) contains neither a zero
nor the pole of \(\zeta\). Choose its holomorphic logarithm to
agree with the Euler logarithm near the centre. Polynomial growth
of \(\zeta\) on a fixed strip and the Borel--Carath\'eodory theorem
give \(\log\zeta(z)\ll_{\sigma_0}\log(3+|t|)\) on the disc of
radius \(R_1\). On the disc of radius \(1/2\), the absolutely
convergent Euler logarithm is bounded. Hadamard's three-circles
theorem therefore yields
\[
  |\log\zeta(\sigma+it)|
  \ll_{\sigma_0}\{\log(3+|t|)\}^{\kappa},
  \qquad
  \kappa=\frac{\log(2r)}{\log(2R_1)}<1,
\]
for \(\sigma_0\leq\sigma\leq2\). Exponentiating proves the
bounds for \(\zeta\) and \(1/\zeta\). Apply these bounds first
with a slightly smaller left endpoint and use Cauchy's formula
on fixed-radius discs to obtain the derivative bounds. Bounded
ranges of \(t\) are covered by compactness away from the pole.
\end{proof}

On any bounded part of the strips below, we use compactness away
from \(z=1\), and the Laurent expansion at \(z=1\) when a pole
is crossed. Thus the lemma also supplies the bounds needed on
all the unbounded contours. Recall that \(a=1/100\); in particular,
\(1-2a>7/8\).

\subsection{The two residues}

Define
\[
  \mathcal I_L(\alpha,\beta)
  =\frac1{(2\pi i)^2}\int_{(a)}\int_{(a)}
     e^{L(\alpha w_1+\beta w_2)}
     \mathcal R_L(w_1,w_2)\,dw_1\,dw_2.
\]
For \(\gamma=\min\{\alpha,\beta\}\), put
\begin{equation}\label{eq:evaluated-model}
  \mathcal J_\vartheta(\alpha,\beta)
  =2\alpha\beta\gamma
   +\left(\frac\vartheta2-\alpha-\beta\right)\gamma^2
   +\frac{\gamma^3}{3}.
\end{equation}

\begin{proposition}\label{prop:residue-evaluation}
Uniformly for \(\alpha,\beta\in[1/2,1]\) and
\(\vartheta\in[3,5]\),
\begin{equation}\label{eq:residue-evaluation}
  \mathcal I_L(\alpha,\beta)
  =L^3\mathcal J_\vartheta(\alpha,\beta)+O_g(L^2).
\end{equation}
\end{proposition}

\begin{proof}
By symmetry we may assume \(\alpha\geq\beta\). For each fixed
\(w_2\) on \(\Re w_2=a\), move the \(w_1\)-contour to
\(\Re w_1=-2a\). The arithmetic factor remains in the region of
Lemma~\ref{lem:arithmetic}, and the reciprocal-zeta factors are
holomorphic by the fixed zero-free half-plane. The only possible
poles crossed are \(w_1=0\) and \(w_1=-w_2\), of orders at most
two and three, respectively. The order at zero also follows from
the expansions below.

Lemma~\ref{lem:fixed-strip-growth} and the bounded arithmetic
factor and its derivatives imply, on the new pair of lines,
\[
  |\mathcal R_L(-2a+it_1,a+it_2)|
  \ll_g (L+1)
       (1+|t_1|)^{-5/2}(1+|t_2|)^{-5/2}.
\]
Indeed, each zeta factor and each fixed derivative costs an
arbitrarily small power of its height, which can be absorbed into
the denominators \(w_1^3w_2^3\). The new double integral is thus
\[
  \ll_g L e^{L(-2a\alpha+a\beta)}
  \ll_g L e^{-aL/2}.
\]
For fixed \(w_2\), the same bounds make the horizontal edges tend
to zero. The new double integral and both integrated residues are
absolutely convergent, as shown here and below, so Fubini's theorem
then permits integration of the one-variable residue identity
in \(w_2\).

We first bound the residue on \(w_1+w_2=0\). Write
\(w=w_2=a+it\) and \(u=w_1+w_2\), and set
\[
  \mathcal P(v,u,w)
  =\frac{\widetilde g(1+v)\mathfrak A(1+v,u-w,w)}
    {(u-w)^3w^3
     \zeta(1+v+u-w)^2\zeta(1+v+w)^2}.
\]
For \(|u|,|v|\leq a/4\), this factor is holomorphic in \(u,v\).
For every fixed \(\epsilon>0\), it and all the derivatives needed
below are
\[
  \ll_{g,\epsilon}(1+|t|)^{-6+\epsilon},
\]
by Lemmas~\ref{lem:arithmetic} and~\ref{lem:fixed-strip-growth}.
The definition of \(\mathcal R_L\) gives
\[
\begin{aligned}
  \mathcal R_L(u-w,w)
  ={}&\bigl\{(\vartheta L+2\gamma_0)\mathcal P(0,u,w)
          +\partial_v\mathcal P(0,u,w)\bigr\}\zeta(1+u)^2\\
     &+\mathcal P(0,u,w)\frac{d}{du}\zeta(1+u)^2.
\end{aligned}
\]
Consequently its principal part at \(u=0\) has the form
\[
  \frac{C_3(w)}{u^3}
  +\frac{C_2(w;L)}{u^2}
  +\frac{C_1(w;L)}{u},
\]
where, on taking a fixed sufficiently small \(\epsilon\),
\[
\begin{aligned}
  |C_3(w)|&\ll_g(1+|t|)^{-11/2},\\
  |C_1(w;L)|+|C_2(w;L)|
  &\ll_g(L+1)(1+|t|)^{-11/2}.
\end{aligned}
\]
The residue after multiplication by the exponential is
\[
  e^{-(\alpha-\beta)Lw}
  \left\{C_1(w;L)+\alpha L C_2(w;L)
                  +\frac{(\alpha L)^2}{2}C_3(w)\right\}.
\]
Since \(|e^{-(\alpha-\beta)Lw}|\leq1\), its integral on
\(\Re w=a\) is \(O_g(L^2)\), uniformly including
\(\alpha=\beta\).

It remains to evaluate the residue at \(w_1=0\). Put
\begin{equation}\label{eq:logarithmic-factor}
  K(w)=\partial_{w_1}\mathfrak A(1,0,w)
           +2\frac{\zeta'}{\zeta}(1+w)-2\gamma_0.
\end{equation}
For fixed \(\Re w=a\), the axis identities
\eqref{eq:arithmetic-axes} and Laurent expansion at zero give
\[
\begin{aligned}
  \mathcal H(0;x,w)&=x^2+K(w)x^3+O_w(x^4),\\
  \partial_v\mathcal H(0;x,w)
    &=2x+\bigl\{2K(w)-2\gamma_0+\widetilde g'(1)\bigr\}x^2
      +O_{g,w}(x^3).
\end{aligned}
\]
Substitution into~\eqref{eq:residue-density} therefore gives
\begin{equation}\label{eq:axis-residue}
\begin{aligned}
 &\mathop{\rm Res}_{w_1=0}
     e^{L(\alpha w_1+\beta w)}\mathcal R_L(w_1,w)\\
 &\qquad=
   e^{\beta Lw}
   \frac{(2\alpha+\vartheta)L+\widetilde g'(1)+2K(w)}{w^3}.
\end{aligned}
\end{equation}
The function \(K\) in~\eqref{eq:logarithmic-factor} is meromorphic in the strip
\(-a\leq\Re w\leq a\), with its only pole at \(w=0\).
By~\eqref{eq:arithmetic-axes},
\(\partial_{w_1}\mathfrak A(1,0,0)=0\). Hence
\[
  K(w)=-\frac2w+O(w)\qquad(w\to0).
\]
Move the integral of~\eqref{eq:axis-residue} from
\(\Re w=a\) to \(\Re w=-a\). The growth lemma bounds
\(K(w)\) by an arbitrarily small power of the height on these
lines. The new integral is therefore \(O_g(Le^{-a\beta L})\),
and the horizontal edges vanish. The residue at zero is
\[
  L^3\left\{
      \left(\alpha+\frac\vartheta2\right)\beta^2
                       -\frac23\beta^3\right\}+O_g(L^2).
\]
Here the pole \(-4/w^4\) contributes
\(-2\beta^3L^3/3\), the coefficient
\((2\alpha+\vartheta)L/w^3\) contributes the other cubic term,
and \(\widetilde g'(1)/w^3\) and the remaining analytic part of
\(K\) contribute \(O_g(L^2)\).
Combining this calculation with the bound for the residue on
\(w_1+w_2=0\) proves~\eqref{eq:residue-evaluation}, since
\(\gamma=\beta\) under the assumed ordering.
\end{proof}

\begin{proof}[Proof of Theorem~\ref{thm:kernel}]
Lemma~\ref{lem:single-residue} and
Proposition~\ref{prop:residue-evaluation} give
\[
  \mathfrak C_{s,u}(Y,R)
  =\frac{4Y}{L}\mathcal J_\vartheta(\alpha,\beta)
     +O_g\!\left(\frac{Y}{L^2}\right).
\]
Elementary integration rewrites~\eqref{eq:evaluated-model} as
\[
  \mathcal J_\vartheta(\alpha,\beta)
  =2\int_0^\gamma(\alpha-r)(\beta-r)\,dr
     +\int_0^\gamma r(\vartheta-r)\,dr.
\]
Using \(G_s'(x)=2(x-s)_+\) and
\(G_s''(x)=2\one_{(s,1)}(x)\) almost everywhere, substitute
\(r=1-x\) in~\eqref{eq:bilinear-form} to obtain
\[
  \mathcal B_\vartheta(G_s,G_u)
      =4\vartheta\mathcal J_\vartheta(\alpha,\beta).
\]
Since \(\log Y=\vartheta L\) and \(3\leq\vartheta\leq5\),
this proves the theorem and~\eqref{eq:explicit-kernel}.
\end{proof}

\section{Differences of shifted weights and small prime factors}
\label{sec:consequences}

\begin{proof}[Proof of Corollary~\ref{cor:shifted-profile}]
Expand
\[
 W_{H_t,R}
 =\Lambda_{0,R}^2-2\Lambda_{0,R}\Lambda_{t,R}
  +\Lambda_{t,R}^2
\]
and apply Theorem~\ref{thm:kernel} to the three products. By
bilinearity,
\[
 \mathcal B_\vartheta(G_0,G_0)
 -2\mathcal B_\vartheta(G_0,G_t)
 +\mathcal B_\vartheta(G_t,G_t)
 =\mathcal B_\vartheta(H_t,H_t).
\]
The three uniform errors are still \(O_g(Y/(\log Y)^2)\), which
proves~\eqref{eq:shifted-profile-asymptotic}.

The piecewise form of the difference profile is
\[
  H_t(x)
  =
  \begin{cases}
    0, & x\leq0,\\
    x^2, & 0<x\leq t,\\
    2tx-t^2, & x>t.
  \end{cases}
\]
Almost everywhere on \((0,1)\),
\[
  H_t'(x)
  =
  \begin{cases}
    2x, & 0<x<t,\\
    2t, & t<x<1,
  \end{cases}
  \qquad
  H_t''(x)=2\,\one_{(0,t)}(x).
\]
Consequently,
\[
\begin{aligned}
  \int_0^1\bigl(H_t'(x)\bigr)^2\,dx
  &=
  \int_0^t4x^2\,dx+\int_t^14t^2\,dx\\
  &=
  4t^2-\frac83t^3,
\end{aligned}
\]
whereas
\[
\begin{aligned}
  &\int_0^1
    \bigl(H_t''(x)\bigr)^2
    (1-x)(\vartheta-1+x)\,dx\\
  &\qquad=
  4\int_0^t
    \bigl\{(\vartheta-1)+(2-\vartheta)x-x^2\bigr\}\,dx\\
  &\qquad=
  4(\vartheta-1)t
  +2(2-\vartheta)t^2
  -\frac43t^3.
\end{aligned}
\]
Substitution into~\eqref{eq:bilinear-form} gives
\[
  \mathcal Q_\vartheta(H_t)
  =
  \vartheta\left\{
    4(\vartheta-1)t+(12-2\vartheta)t^2
    -\frac{20}{3}t^3
  \right\},
\]
which is~\eqref{eq:shifted-profile-main-term}. The weights in
\(\mathcal B_\vartheta(H_t,H_t)\) are nonnegative for
\(3\leq\vartheta\leq5\). The displayed polynomial therefore
gives
\[
 0\leq\mathcal Q_\vartheta(H_t)\ll t
 \qquad(0\leq t\leq1/2),
\]
uniformly in \(\vartheta\). Combining this with
\eqref{eq:shifted-profile-asymptotic} proves
\eqref{eq:shifted-profile-bound}.
\end{proof}

\subsection{Integers divisible by a small prime}

The divisor-pairing identity from the introduction gives the
following estimate.

\begin{corollary}
\label{cor:small-prime}
Uniformly for \(Y^{1/5}\leq R\leq Y^{1/3}\) and primes
\(p\leq R^{1/2}\), one has
\[
\begin{aligned}
  &\sum_{\substack{n\geq1\\p\nmid n}}
    g\!\left(\frac nY\right)\mu^2(n)2^{\om(n)}
    \Lambda_{0,R}(pn)^2\\
  &\qquad\ll_g
  \left(\frac{\log p}{\log R}+\frac1{\log Y}\right)
  \frac{Y}{\log Y}
  \ll_g
  \frac{(1+\log p)Y}{(\log Y)^2}.
\end{aligned}
\]
\end{corollary}

\begin{proof}
Put \(t_p=\log p/\log R\), so \(0\leq t_p\leq1/2\).  If
\(p\nmid n\) and \(x_e=\log(R/e)/\log R\), pairing the divisors
\(e\) and \(pe\) of \(pn\) and using \(\mu(pe)=-\mu(e)\) gives
\[
\begin{aligned}
  \Lambda_{0,R}(pn)
  &=\sum_{e\mid n}\mu(e)
    \bigl\{G_0(x_e)-G_0(x_e-t_p)\bigr\}\\
  &=\sum_{e\mid n}\mu(e)H_{t_p}(x_e).
\end{aligned}
\]
Here the positive-part convention makes the cutoff indicators
redundant.  Hence
\[
  \Lambda_{0,R}(pn)^2=W_{H_{t_p},R}(n).
\]
All summands on the left of~\eqref{eq:shifted-profile-bound} are
nonnegative.  Restricting that sum to \(p\nmid n\) and using the
preceding identity proves the first bound.  The second follows from
\(\log R\asymp\log Y\).
\end{proof}

For an application to moduli \(pn\) of size \(X\), the preceding
corollary gives
\[
 \sum_{\substack{n\geq1\\p\nmid n}}
 g\!\left(\frac{pn}{X}\right)\mu^2(n)2^{\om(n)}
 \Lambda_{0,R}(pn)^2
 \ll_g \frac{(1+\log p)X}{p\{\log(X/p)\}^2},
\]
provided that
\[
 (X/p)^{1/5}\leq R\leq(X/p)^{1/3},
 \qquad p\leq R^{1/2}.
\]
The size of the contribution after summing over primes depends on
the prime cutoff. These are the conditions and the bound used in
the small-prime argument of Part~II.

\section*{Acknowledgments}
This work was supported by the National Natural Science Foundation of China (Grant Nos.~12171311 and~12671012). Y.X. acknowledges financial support from the China Scholarship Council (CSC), as well as the support and hospitality of the School of Mathematics and Statistics at UNSW Sydney through its PhD Support Scheme.

\end{document}